\documentclass[12pt]{article}
\usepackage{tikz}

\usepackage{amsmath,amscd,amssymb,amsthm}
\usepackage{csquotes}
\usepackage{comment}
\usepackage{tabularx}
\usepackage{subfigure}
\usepackage{hyperref}
\usepackage{multirow}
\newtheorem{theorem}{Theorem}[section]
\newtheorem{prop}{Proposition}[section]

\newtheorem{definition}[theorem]{Definition}
\newtheorem{example}[theorem]{Example}
\newtheorem{remark}[theorem]{Remark}

\newtheorem{fact}[theorem]{Fact}
\numberwithin{equation}{section}

\begin{document}
	\title{Modular Surfaces in Lorentz-Minkowski 3-Space: Curvature and Applications}

	\author{Siddharth Panigrahi, Subham Paul, Rahul Kumar Singh, Priyank Vasu}
	
	\date{}

	\maketitle
	
	\vspace{-1cm}
	\begin{center}
	\textit{Department of Mathematics, Indian Institute of Technology Patna,}\\
	\textit{Bihta, Patna-801106, Bihar, India}
	\end{center}

	\vspace{-0.5cm}
	\begin{center}
    \texttt{siddharth\_2121ma12@iitp.ac.in}\\
	\texttt{subham\_2021ma25@iitp.ac.in} \\
  \texttt{rahulks@iitp.ac.in}\\
        \texttt{priyank\_2121ma16@iitp.ac.in}
	\end{center}

	\begin{abstract} 
	In this paper, we study the relation of the sign of the Gaussian and mean curvature of modular surfaces in Lorentz-Minkowski $3$-space to the zeroes of the associated complex analytic functions and its derivatives. Further, we completely classify zero Gaussian curvature modular surfaces. Next we show non-existence of non-planar maximal modular surfaces, characterize CMC modular surfaces, analyze asymptotic behaviour of Gaussian curvature of complete modular graphs and the Hessian of their height functions and lastly as application, demonstrate how modular surfaces can be realised as integral surfaces of some conformal field theories and non-linear sigma models. 
	\end{abstract}

    \vspace{1em}
	
    \vspace{1em}

\section{Introduction}

A modular surface in $\mathbb R^3$ is defined as the graph of the modulus of a complex analytic function. 
The concept of modular surfaces was very likely first introduced by E. Maillet (see \cite{maillet} and \cite{reimerdes1911niveau}).
The differential geometric properties of 
modular surfaces  in the Euclidean space $\mathbb E^3:=(\mathbb R^3, dx^2+dy^2+dt^2)$, where $(x,y,t)$ are the cartesian coordinates of $\mathbb R^3$, were first studied by J. Zaat in \cite{jaat} (see \cite{Kreyszig} also).
 In \cite{ullrich1951amountareas}, E. Ullrich classified these surfaces based on the sign of their Gaussian curvature function. To the best of our knowledge, modular surfaces have not been explored much since then. 
 
 Modular surfaces are known to be useful in studying some important meromorphic functions, such as Riemann's zeta functions and Euler's gamma function, etc. They also find its applications in many different areas of mathematics, such as complex dynamical systems, Laplace and complex Fourier transform, to name a few (see \cite{Wegert} and references therein).


In this paper, we study the modular surfaces in the Lorentz-Minkowski space $\mathbb E_1^3:=(\mathbb R^3, dx^2+dy^2-dt^2)$. A surface in $\mathbb E_1^3$ can have different causal characters: timelike, spacelike, lightlike, and they can be of mixed type too (see \cite{Rafael}), and as because naturally, degeneracy arises in these surfaces, studying the geometric properties of such surfaces is somewhat complicated compared to the surfaces in $\mathbb E^3$. 
\par Firstly, we have shown that the sign of Gaussian and mean curvatures at a point and in its neighbourhood depends on the order of zeroes of the generating analytic function and its derivatives. These results have been compiled in three tables (see Table \ref{tab:my_label1}, Table \ref{tab:my_label2} and Table \ref{newtable_new}). Along with this, we have also given a complete classification of zero Gaussian curvature modular surfaces and zero mean curvature (ZMC) modular surfaces in $\mathbb{E}_1^3$. In particular, we have shown that \textit{the only examples of ZMC modular surfaces in $\mathbb E_1^3$ are planar sections} (see Theorem \ref{mainthm}).  A similar argument can be used to prove the non-existence of ZMC (minimal) modular surfaces in $\mathbb E^3$. Lastly using the log-harmonicity of modular height functions, we give a characterization of \enquote{CMC modular surfaces}, upto a change of parametrization (see Theorem \ref{thm5.1}), give an asymptotic analysis of the Gaussian curvature of convex modular surfaces (see Propositions \ref{prop5.1} and \ref{prop5.2}), the asymptotic vanishing of the Hessian of their height functions (see Proposition \ref{prop5.3}) and actualize modular surfaces as local integral surfaces of Liouville and Massless Free Scalar Conformal Field Theories (see Subsections \ref{cft} and \ref{massless}) and Non-Linear Sigma Models (see Subsection \ref{sigma}). Moreover, these solutions are \enquote{analytically easily computable} since the Euler-Lagrange equations in these cases reduce to first order partial differential equations.

\par
This article is arranged as follows. In Section \ref{P}, we have formally introduced the notion of modular surface. We have also shown how the causality of such surfaces in $\mathbb E_1^3$ depends on the first derivative of the generating complex analytic function. Next in Section \ref{CMS}, we obtain the formula for Gaussian curvature of a modular surface in $\mathbb E_1^3$ in terms of its generating function $F$ 
 and its derivatives $F', F''$. This enables us to discuss the sign of Gaussian curvature of a modular surface and its relation to the order of zeroes of $F, F^{\prime}, F^{\prime\prime}$( see Theorem \ref{thimp} and the propositions following it). Also, we completely classify the analytic functions associated with modular surfaces for which the Gaussian curvature vanishes identically (see Theorem \ref{gausszero}). This also identifies a subclass of analytic functions. In Section \ref{meansection}, we examine the mean curvature of modular surfaces and show the non-existence of ZMC modular surfaces in $\mathbb E^3$ and $\mathbb E_1^3$ (see Theorem \ref{mainthm} and the remark following it). Finally, in Section \ref{application}, we use the log-harmonicity of the modular height function to give the aforementioned applications.





\section{Preliminaries}\label{P} 
  Throughout this article, an element $(x,y)$ of $\mathbb{R}^2$ is identified with the complex number $z=x+iy$. The term \textit{domain} stands for an open connected subset of $\mathbb{C}$. We begin by defining the \textit{modular surface (modular graph)}.
\begin{definition}
   A surface in $\mathbb R^3$ parametrized by $M(x,y)= (x, y, h(x,y))$ is said to be modular if $h(x,y)= |F(z)|$, where $z=x+iy\in D$ and $F$ is an analytic function defined on a domain $D$ in $\mathbb C$. Here, $|\;\;|$ denotes the complex modulus.
\end{definition}
\noindent
\textbf{Note:} The function $h$ is called the height function of the modular surface $M(x,y)$ defined on $D$ and $F$ is called its generating function. 

\noindent
Throughout this paper, we will take $F$ as the square of another analytic function $f$ unless otherwise specified, i.e., $F=f^2$ on a domain $D$. This is required as we want our $h$ to be always smooth; for example, $|z|=\sqrt{x^2+y^2}$ is not smooth at the origin. So the following will be the notational convention,
\begin{equation}
    \label{mark1}
    F= u+iv = (a+ib)^2,
\end{equation}
where $f=a+ib$. Clearly, $u, v, a, b$ are all harmonic functions defined on $D$. If $F$ happens to be non-zero at a point, then around that point, the analytic square root of $F$ exists, which ensures the smoothness of $h$ in a neighbourhood of that point. 
\begin{remark}
    If $F, G$ are two non-zero analytic functions on $D$ such that $|F|=|G|$, then from the \textit{Identity Theorem} we conclude that $F= kG$ on $D$, where $k\in\mathbb{C}$ with $|k|=1$. In this case, $F$ and $G$ generate the same modular surface. Let us denote this relation by $F\sim G$. Then, ${\sim}$ will be an equivalence relation on the class of all analytic functions defined on $D$. This shows that given a modular surface defined on $D$ with height function $h$, there exists a unique equivalence class of analytic functions on $D$ such that $h=|F|$, where $F$ is any representative of that class.
\end{remark} Before we proceed any further, we would like to introduce an easily verifiable fact that will be helpful for our discussion.
\begin{fact}(\cite{Weiser})
\label{fact2}
    If $h$ is the height function of a modular surface, then $h(h_{xx}+h_{yy})=h_x^2+ h_y^2$.
\end{fact}
By a classical result from complex analysis (see \cite{Ahlfors}, Page-142), we can conclude that if $h(>0)$ is a height function for some modular surface defined on a simply connected domain $D$ with generating function $F$, then $h=|F(z)|=|e^{H(z)}|$, where $H$ is analytic on $D$. Hence, we have the following fact:
\begin{fact}
\label{fact1}
    A positive function \( h > 0 \) is the height function of a modular surface over a simply connected domain \( D \) if and only if \( \log h \) is harmonic on \( D \). In this case, \( h \) is called a log-harmonic function. Since \( \log h \) is harmonic, it follows that \( h \) is subharmonic on \( D \).

\end{fact}
\noindent
 We observe that if $h$ is the height function of a modular surface, then $h+c$ ($c$ is a real constant) may not be the height function for any modular surface, i.e., if we lift a modular surface by a certain height, it may lose its \enquote{modularity}. For example, $h=e^x=|e^z|$ but $h=e^x+1$ is not the height function of any modular surface.

\noindent

\subsection{Modular surfaces in Lorentz-Minkowski space $\mathbb{E}_1^3$}

 Now, we will recall some basic facts about surfaces in $\mathbb{E}_1^3$.
\begin{definition}
    Let $S$ be a 2-dimensional smooth manifold with or without boundary. An immersion $\Bar{x}: S\rightarrow\mathbb{E}_1^3$ is called spacelike (respectively timelike, lightlike) if for all $p\in S$, the metric $d\Bar{x}_p^*(,)$ is positive definite (respectively a metric with index 1, a degenerate metric), where
    \begin{equation*}
      d\Bar{x}_p^*(u,v)=\langle d\Bar{x}_p(u),\hspace{0.1cm}d\Bar{x}_p(v)\rangle;\forall u,v\in T_p(S).
    \end{equation*}
\end{definition}
The spacelike, timelike or lightlike nature of a surface is called the causal character of the surface. So we see that for the immersion generated by the height function $h$ as $M(x,y)= (x, y, h(x,y));\forall (x,y)\in D$, the matrix of the induced metric with respect to the basis $\{M_x=(1,0,h_x), M_y=(0,1,h_y)\}$ is
\begin{equation*}
   \left(g_{ij}\right)_{i,j=1,2}= \begin{bmatrix}
1-h_x^2 & -h_xh_y \\
-h_xh_y &  1-h_y^2
\end{bmatrix},
\end{equation*} where $g_{11}=\langle M_x,\hspace{0.1cm} M_x\rangle,\hspace{0.1cm} g_{12}=g_{21}=\langle M_x,\hspace{0.1cm} M_y\rangle,\hspace{0.1cm} g_{22}=\langle M_y,\hspace{0.1cm}M_y\rangle$ and the determinant of $\left(g_{ij}\right)_{i,j=1,2}$ is $1-h_x^2-h_y^2=1-|\nabla h|^2$. Thus, the immersion is spacelike if $|\nabla h|^2=h_x^2+h_y^2<1$, timelike if $|\nabla h|^2=h_x^2+h_y^2>1$ and lightlike or degenerate if $|\nabla h|^2=h_x^2+h_y^2=1$.
\par
\noindent
 From Fact \ref{fact2}, we have the following result:
\begin{equation}
\label{eq1}
    h(h_{xx}+h_{yy})=h_x^2+h_y^2=4(a^2+b^2)(a_x^2+b_x^2).
\end{equation}
Therefore from equation \eqref{eq1} we have
\begin{equation}
\label{eq2}
    |\nabla h|^2=h_x^2+h_y^2=4|ff^{\prime}|^2=|F^{\prime}|^2.
\end{equation}
Hence, the immersion is spacelike if $|F^{\prime}|<1$, timelike if $|F^{\prime}|>1$ and lightlike if $|F^{\prime}|=1$. Therefore, in a sufficiently small neighbourhood of a zero of $F=f^2$ (say $z_0$), the causal character of the surface will be spacelike as $|F^{\prime}(z_0)|=|2f(z_0)f^{\prime}(z_0)|=0<1$.\par
Also, the expressions of Gaussian curvature ($K$) and mean curvature ($H$) for the modular surface in $\mathbb E_1^3$ are given as follows:
\begin{equation}
\label{eq3}
    K=-\frac{h_{xx}h_{yy}-h_{xy}^2}{(1-h_x^2-h_y^2)^2},
\end{equation}
\begin{equation}
    \label{eq7}
    H=-\frac{(1-h_y^2)h_{xx}+2h_xh_yh_{xy}+(1-h_x^2)h_{yy}}{2|1-h_x^2-h_y^2|^{3/2}}.
\end{equation}
For more details on the geometry of surfaces in $\mathbb E_1^3$ (see \cite{Rafael}).

\section{Construction of Modular surfaces and Sign of Gaussian Curvature }\label{CMS}
  In this section, we derive the formula for Gaussian curvature of a non-degenerate modular surface (i.e., $|F'|\neq 1$) in terms of $F$ and its complex derivatives, and then using this, we will discuss the behaviour of sign of Gaussian curvature in a neighbourhood of a point.

\subsection{Gaussian curvature}
Recall from the equation \eqref{mark1} $h=|F|=|f^2|=a^2+b^2$. Therefore, we have
\begin{align*}
   h_{xx}h_{yy}-h_{xy}^2&=2[(aa_{xx}+bb_{xx})+(a_x^2+b_x^2)]\cdot2[(a_x^2+b_x^2)-(aa_{xx}+bb_{xx})]-[2(ba_{xx}-ab_{xx})]^2\\
    &=4[(a_x^2+b_x^2)^2-(a^2+b^2)(a_{xx}^2+b_{xx}^2)]\\
    &=4(|f^{\prime}|^4-|ff^{\prime\prime}|^2).
\end{align*}
Using the above expression and equation \eqref{eq2} in \eqref{eq3} we get
\begin{equation}
\label{eq4}
    K=-\frac{4\left(|f^{\prime}|^4-|ff^{\prime\prime}|^2\right)}{\left(1-4|ff^{\prime}|^2\right)^2}.
\end{equation}
Now if we use $|f^{\prime}|=|\frac{F^{\prime}}{2\sqrt{F}}|$ and $|f^{\prime\prime}|=|\frac{2FF^{\prime\prime}-F^{{\prime}^2}}{4F^{3/2}}|$ as $F=f^2$ in equation \eqref{eq4}, we get
\begin{equation}
    \label{eq5}
    K=-\frac{|F^{\prime}|^4-|2FF^{\prime\prime}-F^{{\prime}^2}|^2}{4|F|^2(1-|F^\prime|^2)^2},
\end{equation}
where $F\neq 0$. We can simplify it further to obtain
\begin{equation}
    \label{eq6}
    K=-\frac{|F^{\prime\prime}|^2\left[\operatorname{Re}\left(\frac{F^{{\prime}^2}}{FF^{\prime\prime}}\right)-1\right]}{\left(1-|F^{\prime}|^2\right)^2},
\end{equation}
where $F^{\prime\prime}\neq0$.

\begin{remark}
   If $F=0$ at a point, then from equation \eqref{eq4} we get $K=-4|f^{\prime}|^4$.
\end{remark}

\subsection{Construction of Modular surfaces}
From now on, we will only talk about non-degenerate surfaces. It is evident from the formula of Gaussian curvature in equation \eqref{eq6} that the expression $\frac{F^{{\prime}^2}}{FF^{\prime\prime}}$ plays an important role there, especially in determining the sign of Gaussian curvature. Ullrich in his paper (\cite{ullrich1951amountareas}) has established the relationship between $F$ and $\frac{F^{{\prime}^2}}{FF^{\prime\prime}}$ by solving $F$ in terms of $1-\frac{FF^{\prime\prime}}{F^{{\prime}^2}}$. He used that relation to provide a procedure for constructing modular surfaces in $\mathbb E^3$. We adopt that idea here for constructing modular surfaces in $\mathbb {E}_1^3$ with positive or negative Gaussian curvature everywhere. For that we need to first introduce $\alpha$ and $\beta$ as follows:
\begin{align}
\label{eqa16}
    \alpha&=\frac{F^{{\prime}^2}}{FF^{\prime\prime}},\\
 \label{eqa17}   
     \beta&=1-\frac{1}{\alpha}=\frac{F^{{\prime}^2}-FF^{\prime\prime}}{F^{{\prime}^2}}.
\end{align}
\par
We will assume throughout this section that $\alpha, \beta$ are well-defined in the domain, i.e., $F, F^{\prime}, F^{\prime\prime}$ never vanish in the domain. One can clearly observe from equation \eqref{eqa17} that $\beta$ can never be $1$ in the domain. Now if $t=\frac{1}{2}-\beta$, then $\operatorname{Re}\,(\alpha)=\frac{1+2\operatorname{Re}(t)}{\frac{1}{2}+2\operatorname{Re}(t)+2|t|^2}$. Therefore from  equation \eqref{eq6} we have
\begin{align}
    \label{eqa18}
    \operatorname{Re}\,(\alpha)<1&\iff |t|>\frac{1}{2}\iff K>0,\\
    \label{eqa19}
    \operatorname{Re}\,(\alpha)=1&\iff |t|=\frac{1}{2}\iff K=0,\\
    \label{eqa20}
    \operatorname{Re}\,(\alpha)>1&\iff |t|<\frac{1}{2}\iff K<0.
\end{align}

Let us now consider that our domain $D$ in $\mathbb{C}$ is simply connected and contains the origin. Then, Ullrich has shown that 
\begin{equation}
    \label{eq21}
    F(z)=F(z;k)=F(0)exp\int_0^z\frac{ds}{k+\int_0^s\beta(\tau)d\tau},
\end{equation}
where $k=\frac{F(0)}{F^{\prime}(0)}$.
\par

Also it is easy to prove that if $\beta$ is any analytic function on a simply connected domain $D$ containing the origin such that $\beta$ is never $1$ and $k$ is a complex constant such that $k+\int_0^s\beta(s)ds$ never vanishes on the domain $D$, then $F$ determined by \eqref{eq21} will satisfy the equation \eqref{eqa17}.
\\
Therefore, we have the following theorem characterising all modular surfaces for which $F, F^{\prime}, F^{\prime\prime}$ never vanish on the domain.  
\begin{theorem}(\cite{ullrich1951amountareas})\label{mtcmodular}
    Let $M$ be a modular surface defined on a simply connected domain $D$ containing the origin and induced by the analytic function $F$ such that $F, F^{\prime}, F^{\prime\prime}$ never vanish. Then 
    \begin{equation}
    \label{eq22}
    F(z)=F(0)exp\int_0^z\frac{ds}{k+\int_0^s\beta(\tau)d\tau},
\end{equation}
where $k=\frac{F(0)}{F^{\prime}(0)}$, $\beta=\frac{F^{{\prime}^2}-FF^{\prime\prime}}{F^{{\prime}^2}}$ and conversely.
    \end{theorem}

The following example shows how to construct a modular surface in $\mathbb{E}_1^3$ with Gaussian curvature having a particular sign throughout the domain by choosing appropriate $\beta$.
\begin{example}
    Consider $\beta(z)=z+\frac{1}{2}$, defined in the domain $D=\{z\in\mathbb{C}:|z|<\frac{1}{2}\}$ and $k=-1$ in the equation \eqref{eq22}. Then $\beta$ is never $1$ as well as $k+\int_0^z\beta(\tau)d\tau$ never vanishes in $D$. So if we take $F(0)=(-\frac{1}{2})^{\frac{2}{3}}$, then from equation \eqref{eq22} we get $F(z)=\left(\frac{z-1}{z+2}\right)^{\frac{2}{3}}$. This generates a modular surface with everywhere negative Gaussian curvature as $|\beta-\frac{1}{2}|<\frac{1}{2}$ in $D$.

\end{example}
 The next theorem determines all modular surfaces in $\mathbb{E}_1^3$ with identically zero Gaussian curvature.

 \begin{theorem}\label{gausszero}
     Let $M$ be a modular surface in $\mathbb{E}_1^3$ induced by $F$ such that the Gaussian curvature of the surface vanishes identically. Then 
     $$F(z)=exp(mz+n)\; \text{or,}\; (mz+n)^{1+il},\;\text{where}\;m,n\in \mathbb{C}; l\in \mathbb{R}.$$
\end{theorem}

\begin{proof}
    From the expressions of $K$ in equations \eqref{eq5} and \eqref{eq6} we can say that $K\equiv 0$ implies $F^{\prime\prime}\equiv 0$ or $\operatorname{Re}\,(\frac{F^{{\prime}^2}}{FF^{\prime\prime}})-1=0$.\\
    If $F^{\prime\prime}\equiv 0$, then $F(z)=mz+n$ for some complex constants $m,n$.\\
    Now let $\operatorname{Re}\,(\frac{F^{{\prime}^2}}{FF^{\prime\prime}})-1=0$. This implies $\operatorname{Re}\,(\alpha)=1$, where $\alpha$ is defined in the equation \eqref{eqa16} and equivalently from equation \eqref{eqa19}, $|\beta-\frac{1}{2}|=\frac{1}{2}$ on the domain. Therefore, $\beta$ is a constant function, and consequently, $\alpha$ is also so. Hence $\alpha=1+il^{\prime}$, where $l^{\prime}$ is a real constant.
    \\
    \textbf{Case 1:} Let $l^{\prime}=0$. Then $\frac{F^{\prime\prime}}{F^{\prime}}=\frac{F^{\prime}}{F}$ and therefore integrating twice we get $F(z)=exp(mz+n),$ where $m,n$ are complex constants.
    \\

    \textbf{Case 2:} Let $l^{\prime}\neq 0$. Then
  $$  \begin{aligned}
        &&\frac{F^{{\prime}^2}}{FF^{\prime\prime}}&=\alpha\\
        \implies &&F^{-\frac{1}{\alpha}}F^{\prime}&=m^{\prime} \hspace{0.5cm} (m^{\prime} \text{ is a complex constant})
    \end{aligned}$$
     Therefore, by integrating we get $F(z)=(mz+n)^{\frac{\alpha}{\alpha-1}}$, where $m,n$ are complex constants. Also $\frac{\alpha}{\alpha-1}=1+il$ with $l=-\frac{1}{l^{\prime}}$.\\
     Hence, the proposition.
  
\end{proof}

\subsection{Sign of Gaussian curvature}
In the previous subsection, we have discussed the sign of Gaussian curvature of Modular surfaces at a point where $F,\, F^{\prime},\, F^{\prime\prime}$ are all non-zero. Now, let us discuss what happens when one or more of them vanish at a particular point, and $F$ is not necessarily the square of an analytic function. For this discussion in Euclidean space, one may refer to \cite{Weiser}. We will exclude the case $F(z)=az+b$ from our discussion, where $a,b$ are complex constants, i.e., $F^{\prime\prime}\equiv 0$. We will take the origin $z=0$ as the point under consideration for simplified computations. Now suppose $F$ is an analytic function defined around the origin such that $F(0)\neq0,\, F^{\prime}(0)\neq0$ and $F^{\prime\prime}(0)=\dots=F^{n-1}(0)=0$ with $F^{n}(0)\neq0$, where $n\geq3$. It is important to note that in a small deleted neighbourhood of $z=0$, the function $\omega=\frac{F^{{\prime}^2}}{FF^{\prime\prime}}$ is analytic and hence Gaussian curvature can be evaluated using the equation \eqref{eq6}.   Now, the next Theorem throws light on the behaviour of the sign of Gaussian curvature of the modular surface induced by $F$ in a sufficiently small neighbourhood of the origin. 

\begin{theorem}
\label{thimp}
    If $F(0)\neq0,\, F^{\prime}(0)\neq0$ and $F^{\prime\prime}(0)=\dots=F^{n-1}(0)=0$ with $F^{n}(0)\neq0$, where $n\geq3$, then the Modular surface induced by $F$ will have $n-2$ regions of positive Gaussian curvature,  $n-2$ regions of negative Gaussian curvature and $2n-4$ curves of zero Gaussian curvature each emanating from the origin separating all these regions in a deleted neighbourhood of $z=0$, and at $z=0$, the Gaussian curvature vanishes.
\end{theorem}
\begin{proof}
    From equation \eqref{eq5} we can say that $K(0)=0$. From the hypothesis it is clear that $\omega=\frac{F^{{\prime}^2}}{FF^{\prime\prime}}$ has a pole of order $n-2$ at the origin. Then $\frac{1}{\omega}=\frac{FF^{\prime\prime}}{F^{{\prime}^2}}$ has a zero of order $n-2$ at the origin. Therefore $\frac{1}{\omega}=z^{n-2}\phi(z)$, where $\phi$ is analytic at the origin such that $\phi(0)\neq0$. Now, since $\phi$ is non-zero in some neighbourhood of $z=0$, we can write $\phi(z)=\psi(z)^{n-2}$ in that neighbourhood and $\psi$ is analytic. Hence $\frac{1}{\omega(z)}=t(z)^{n-2}$, where $t(z)=z\psi(z)$. Then $t(0)=0$ and $t^{\prime}(0)=\psi(0)\neq0$. Now, from equation \eqref{eq6}, we know that $K$ vanishes if and only if $\operatorname{Re}\,(\omega)=1$. Now if $\omega=1+iy (y\in\mathbb{R})$, then $\frac{1}{\omega}=\frac{1}{1+y^2}+i\frac{-y}{1+y^2}=\frac{1}{\sqrt{1+y^2}}e^{i\theta(y)}$, with $\theta(y)\in(-\frac{\pi}{2},\frac{\pi}{2})$. Now, this will trace out a circular arc passing through the origin in $\frac{1}{\omega}-$plane. Since $\frac{1}{\omega}=t^{n-2}$, each point on that circular arc in $\frac{1}{\omega}$ plane in a small neighbourhood of the origin will have $n-2$ preimages in the $t$-plane. Again since $t^{\prime}(0)\neq0$, $t(z)=z\psi(z)$ is injective around $z=0$, each of those $n-2$ points will have exactly one preimage in the $z$-plane. Clearly, at each of these points, Gaussian curvature will vanish.\par
    Now $\theta(y)\rightarrow-\frac{\pi}{2}$ as $y\rightarrow+\infty$ and $\theta(y)\rightarrow\frac{\pi}{2}$ as $y\rightarrow-\infty$. Also $t=\left(\frac{1}{\omega}\right)^{\frac{1}{n-2}}=\frac{1}{(1+y^2)^{\frac{1}{2n-4}}}e^{i\frac{(\theta(y)+2k\pi)}{n-2}}$; $k=0,...,n-3$. For $\theta(y)\rightarrow+\frac{\pi}{2}$ and $\theta(y)\rightarrow-\frac{\pi}{2}$, this will give $(n-2)+(n-2)=2n-4$ curves emanating from the origin in $t$-plane. Hence, all these in a sufficiently small neighbourhood of the origin $z=0$ will have $2n-4$ preimage curves in the $z$- plane. These are the curves of zero Gaussian curvature. Similarly, the regions of positive Gaussian curvature ($\operatorname{Re}\,(\omega)<1$) and negative Gaussian curvature ($\operatorname{Re}\,(\omega)>1$) are each mapped on $n-2$ alternate regions bounded by the $2n-4$ curves of zero Gaussian curvature.
\end{proof}

\begin{remark}
    The tangent rays to the $t$-images of these zero curvature curves emanating from the origin in the $t$-plane will be $e^{i\frac{\pm\pi/2+2k\pi}{n-2}}; k=0,...,n-2$. From this, we can conclude that since $t^{\prime}(0)\neq0$, the zero curvature curves in $z$-planes are equispaced.
\end{remark}
 Now the next set of propositions, together with the Theorem \ref{thimp}, will completely exhaust all the cases where one or more of $F(0), F^{\prime}(0), F^{\prime\prime}(0)$ vanish. However, the technique used to prove them will be similar to that of the proof of the Theorem \ref{thimp} in most cases. Remember, as $F$ may not have an analytic square root around $z=0$, it is not necessary that $F(0)=0 \implies F^{\prime}(0)=0$.
 \begin{prop}
    If $F(0),\,F^{n}(0)\neq0\,(n\geq3)$ and $F^{\prime}(0)=F^{\prime\prime}(0)=0$. Then $K(0)=0$, and in a deleted neighbourhood of $z=0$, the Gaussian curvature is positive.
\end{prop}
 \begin{proof}
     From equation \eqref{eq5} we can say that $K(0)=0$. From the hypothesis we also have that $\omega=\frac{F^{{\prime}^2}}{FF^{\prime\prime}}$ has a removable singularity at $z=0$ such that $lim_{z\rightarrow0}\omega=0$. This implies that $|\frac{F^{{\prime}^2}}{FF^{\prime\prime}}|<1$ in a sufficiently small neighbourhood of $z=0$ and hence in that neighbourhood $\operatorname{Re}\,(\frac{F^{{\prime}^2}}{FF^{\prime\prime}})<1$. Therefore, using equation \eqref{eq6}, we are done.
 \end{proof}
 \begin{prop}
     If $F(0)=F^{\prime\prime}(0)=\dots=F^{n-1}(0)=0$ and $F^{\prime}(0),\,F^{n}(0)\neq0\,(n\geq3)$. Then $K(0)$ is undefined, but in a deleted neighbourhood of $z=0$, there are $n-1$ regions of positive Gaussian curvature, $n-1$ regions of negative Gaussian curvature and $2n-2$ curves of zero Gaussian curvature emanating from the origin separating these regions.
 \end{prop}
 \begin{proof}
     From the equation \eqref{eq5}, we can see that $K(0)$ is undefined. Also $\omega=\frac{F^{{\prime}^2}}{FF^{\prime\prime}}$ has a pole of order $n-1$ at $z=0$. Now, similarly proceeding as in Theorem \ref{thimp}, we conclude the proposition.
 \end{proof}
 \begin{prop}
     If $F(0)=F^{\prime}(0)=F^{\prime\prime}(0)=\dots=F^{n-1}(0)=0$ and $F^{n}(0)\neq0\,(n\geq3)$. Then, $K(0)$ is undefined, but in a deleted neighbourhood of $z=0$, the Gaussian curvature $K$ is negative.
 \end{prop}
 \begin{proof}
     From equation \eqref{eq5}, we can say that $K(0)$ is undefined. Also $\omega=\frac{F^{{\prime}^2}}{FF^{\prime\prime}}$ has a removable singularity at $z=0$ with $lim_{z\rightarrow0}\frac{F^{{\prime}^2}}{FF^{\prime\prime}}=\frac{n^2}{n^2-n}>1$ as $n\geq3$. Therefore in a deleted neighbourhood of $z=0$, $\operatorname{Re}\,(\frac{F^{{\prime}^2}}{FF^{\prime\prime}})>1$ as $\operatorname{Re}\,(\frac{F^{{\prime}^2}}{FF^{\prime\prime}})$ is a continuous function. Therefore, from equation \eqref{eq6}, we can conclude that in this neighbourhood, $K<0$. 
 \end{proof}
 So, in Table \ref{tab:my_label1}, we have compiled all the results proved so far, where $n\geq3$ and $F^{n}(0)\neq0$. Here in Table \ref{tab:my_label1}, $n_+=$ the number of regions of positive Gaussian curvature in a deleted neighbourhood of the point $z=0$ and $n_-=$ the number of regions of negative Gaussian curvature in a deleted neighbourhood of the point $z=0$.
 \begin{table}[ht]
    \centering \small
    \begin{tabular}{ |c|c| }
    \hline
$F(0), \,F^{\prime}(0), \dots, \,F^{n}(0)$  & Conclusion \\ 
\hline
$F(0)\neq0, F^{\prime}(0)\neq0$ and& $K(0)=0$  and\\$F^{\prime\prime}(0)=\dots=F^{n-1}(0)=0$  & $n_+=n-2$, \,$n_-=n-2$ \\ 
\hline
$F(0)\neq 0$ and & $K(0)=0$ and \\
$F^{\prime}(0)=F^{\prime\prime}(0)=0$  & $n_+=1,\,n_-=0$ \\
\hline
$F(0)=F^{\prime\prime}(0)=\dots=F^{n-1}(0)=0$ and & $K(0)$ is undefined and\\$F^{\prime}(0)\neq 0$ & $n_+=n-1$, \,$n_-=n-1$  \\ 
\hline
$F(0)=F^{\prime}(0)=F^{\prime\prime}(0)=\dots=F^{n-1}(0)=0$ & $K(0)$ is undefined and \\& $n_+=0, \,n_-=1$    \\
\hline
\end{tabular}
   
\caption{Sign of Gaussian curvature}
    \label{tab:my_label1}
\end{table}

Now, the next set of propositions deals with the function $F$ and its first and second complex derivatives (i.e., $F^{\prime}, \, F^{\prime\prime}$) only.
 \begin{prop}
     If $F(0),\,F^{\prime\prime}(0)\neq0,\,F^{\prime}(0)=0$, then the Gaussian curvature $K$ is positive in a sufficiently small neighbourhood of the origin $z=0$.
 \end{prop}
 \begin{proof}
     From equation \eqref{eq5} we can see that $K(0)>0$. Now, as $K$ is a continuous function, the statement of the proposition is true.
 \end{proof}
 \begin{prop}
     If $F(0)=0$ and $F^{\prime}(0),\,F^{\prime\prime}(0)\neq0$, then $K(0)$ is undefined. However, in some deleted neighbourhoods of the origin $z=0$, there is one region of positive Gaussian curvature and one region of negative Gaussian curvature separated by two curves of zero Gaussian curvature emanating from the origin.
 \end{prop}
 \begin{proof}
     From the equation \eqref{eq5}, it is clear that $K(0)$ is undefined. Since $F$ is analytic, a deleted neighbourhood of $z=0$ exists, where $F$ never vanishes, and $K$ is well-defined in that deleted neighbourhood. Now from the hypothesis we can say that $\omega=\frac{F^{{\prime}^2}}{FF^{\prime\prime}}$ has a pole of order $1$ at the origin. Now, proceeding similarly as in the proof of Theorem \ref{thimp},  we have the conclusion of the proposition.
 \end{proof}
\begin{prop}
     If $F(0)=F^{\prime}(0)=0$ and $F^{\prime\prime}(0)\neq0$. Then, Gaussian curvature is not defined at $z=0$, but in a deleted neighbourhood of $z=0$, the Gaussian curvature $K$ is positive.
 \end{prop}
 \begin{proof}
     From equation \eqref{eq5}, we can say that $K(0)$ is undefined. Now $\omega=\frac{F^{{\prime}^2}}{FF^{\prime\prime}}$ has a removable singularity at $z=0$ with $lim_{z\rightarrow0}\frac{F^{{\prime}^2}}{FF^{\prime\prime}}=0$. This implies that $|\frac{F^{{\prime}^2}}{FF^{\prime\prime}}|<1$ in a sufficiently small deleted neighbourhood of $z=0$ and hence in that neighbourhood $\operatorname{Re}\,(\frac{F^{{\prime}^2}}{FF^{\prime\prime}})<1$. Therefore, from equation \eqref{eq6}, we can conclude that in this neighbourhood, $K>0$.
 \end{proof}


This concludes our discussion on the sign of Gaussian curvature. Finally, compiling the results of the last three Propositions, we get Table \ref{tab:my_label2} with the same notations as in Table \ref{tab:my_label1}:
\begin{table}[h]
    \centering
    \begin{tabular}{ |c|c|c|c| } 
\hline
$F(0)$ & $F^{\prime}(0)$ & $F^{\prime \prime}(0)$ & Conclusion\\
\hline
$\neq0$ & $=0$ & $\neq0$ & $K(0)>0$ and $n_+=1, \,n_-=0$ \\ 
\hline
$=0$ & $\neq0$ & $\neq0$ &$K(0)$ is undefined and $n_+=1$, \,$n_-=1$  \\
\hline
$=0$ & $=0$ & $\neq 0$ & $K(0)$ is undefined and $n_+=1, \,n_-=0$ \\ 
\hline
\end{tabular}
    \caption{Sign of Gaussian curvature}
    \label{tab:my_label2}
\end{table}

\section{The mean curvature of modular surfaces in $\mathbb{E}_1^3$}\label{meansection}
In this section, we derive the formula for the mean curvature of a modular surface in terms of its generating function $F$,  and its derivatives $F^{\prime},\, F^{\prime\prime}$ and then using it we study the change in sign of mean curvature at a point. Finally, we show the associated generating function for a modular surface whose mean curvature vanishes identically is constant.
\subsection{Mean curvature} To evaluate the expression of the mean curvature, we substitute the value of $h_{xx}+h_{yy}$ from the equation \eqref{eq1} in equation \eqref{eq7} and simplify. Then we get the following expression of $H$ where $F\neq0$ and $F^{\prime}\neq0$:
\begin{equation}
\label{eq8}
    H=-\frac{\frac{h_x^2+h_y^2}{h}\left[1-\frac{\rho}{h_x^2+h_y^2}\right]}{2|1-h_x^2-h_y^2|^{3/2}},
\end{equation}
where
\begin{equation}
    \label{eq9}
    \rho=h_{xx}hh_y^2+h_{yy}hh_x^2-2h_xh_yhh_{xy}.
\end{equation}
Now if we express $\rho$ in terms of $u$ and $v$ we get
\begin{equation}
    \label{eq10}
\rho=|F^{\prime}|^4+Au_{xx}+Bv_{xx},
\end{equation}
where
\begin{align}
\label{eqa11}
    A&=-u(u_x^2-v_x^2)-2vu_xv_x,\\
 \label{eqa12}   
    B&=v(u_x^2-v_x^2)-2uu_xv_x.
\end{align}
Using equations \eqref{eq10} and \eqref{eq2} in \eqref{eq8}, we get 
\begin{equation}
    \label{eq13}
    H=-\frac{1}{2}\frac{|F^{\prime}|^2}{|F|}\cdot\frac{1-|F^{\prime}|^2\left[1+\frac{Au_{xx}+Bv_{xx}}{|F^{\prime}|^4}\right]}{2|(1-|F^{\prime}|^2)|^{3/2}},
\end{equation}
where $|F'|\neq 1$.
We also have 
\begin{equation}
    \label{eq14}
    -|F^{\prime}|^4\operatorname{Re}\left(\frac{FF^{\prime\prime}}{F^{{\prime}^2}}\right)=Au_{xx}+Bv_{xx}.
\end{equation}
So by substituting \eqref{eq14} in \eqref{eq13}, we get the expression of $H$ as follows:
\begin{equation}
    \label{eq15}
   H=-\frac{1}{2}\frac{|F^{\prime}|^2}{|F|}\cdot\frac{1-|F^{\prime}|^2\left[1-\operatorname{Re}\left(\frac{FF^{\prime\prime}}{{F^{\prime}}^2}\right)\right]}{2|(1-|F^{\prime}|^2)|^{3/2}}.
\end{equation}

Next, using this formula for $H$, we can determine the sign of mean curvature when $F\neq 0$ and $F^{\prime}\neq 0$ as well as when $F=F^{\prime}=F^{\prime\prime}=0$. It is to be noted that $F=0$ and $F^{\prime}\neq 0$ will give rise to the case where the surface at that point is not smooth, making mean curvature undefined at that point. For, $F\neq 0$ but $F^{\prime}=0$ case, the mean curvature can be determined to be $0$ from the equations \eqref{eq1}, \eqref{eq2} and \eqref{eq7}. These observations are listed in Table \ref{newtable_new}.

\begin{table}[ht]
    \centering 
    
    \begin{tabular}{ |c|c|c|c|c| } 
\hline
$F$ & $F^{\prime}$ & $F^{\prime \prime}$ & Condition & Conclusion\\
\hline
$\neq0$ & $\neq0$ & $\neq0$ &$\operatorname{Re}\left(\frac{FF^{\prime\prime}}{F^{{\prime}^2}}\right) \begin{aligned}
     \leqslant\\=\\ \geqslant
\end{aligned} \frac{1}{{|F^{\prime}|^2}}-1$&$ H \begin{aligned}
    \geqslant\\=\\ \leqslant
\end{aligned} 0$ \\ 
\hline
$\neq0$ & $\neq0$ & $=0$ & $|F^{\prime}|\begin{aligned}
     >\\<
\end{aligned}1$&$H\begin{aligned}
     >\\<
\end{aligned}0$ \\

\hline
$\neq0$ & $=0$ & $\neq 0$ & -----&$H=0$  \\ 
\hline
$\neq0$ & $=0$ & $=0$ & -----&$H=0$  \\ 
\hline
$=0$ & $\neq0$ & $\neq 0$ & -----&$H$ is undefined \\ 
\hline
$=0$ & $\neq0$ & $=0$ & -----& $H$ is undefined  \\ 
\hline

$=0$ & $=0$ & $=0$ & -----& $H=0$  \\ 
\hline
\end{tabular}
    \caption{Sign of mean curvature}
    \label{newtable_new}
\end{table}

A spacelike (timelike) ZMC surface in $\mathbb E_1^3$ is called maximal (timelike minimal), and a ZMC surface in $\mathbb E^3$ is called minimal. The theory of Zero mean curvature (ZMC) surfaces in $\mathbb{E}^3$, and $\mathbb{E}_1^3$ is a well-studied area, and therefore, it will be interesting to see their connection to modular surfaces in general. The next theorem shows the connection.

\begin{theorem}\label{mainthm}
    There are no maximal and timelike minimal surfaces in $\mathbb E_1^3$, which can be parametrized as a modular surface over some domain in $xy$-plane  except for plane sections.
\end{theorem}
\begin{proof}
  We know that $F=c$ (constant) will induce a modular maximal surface, which is a plane section. So, let us assume that $F^{\prime}$ is not identically zero. If we use the expression of mean curvature of a modular surface in $\mathbb{E}_1^3$ from the equation \eqref{eq15}, we find that around the points where $F$ and $F^{\prime}$ are non-zero, equations of maximal and timelike minimal surfaces are given by: 
\begin{equation}
 \label{zeroeqn}
    \operatorname{Re} \left(\frac{FF^{\prime\prime}}{{F^{\prime}}^2}\right) = 1- \frac{1}{{|F^{\prime}|}^2},\,\, |F^{\prime}|<1.\hspace{1cm}(\textit{\textbf{maximal}})
\end{equation}
\begin{equation}
 \label{zeroeqntimelike}
    \operatorname{Re} \left(\frac{FF^{\prime\prime}}{{F^{\prime}}^2}\right) = 1- \frac{1}{{|F^{\prime}|}^2}, \,\,|F^{\prime}|>1.\hspace{1.6cm}(\textit{\textbf{timelike minimal}})
\end{equation}
 We can see that both $\left|\frac{1}{F^{{\prime}^2}}\right|$ and $\log\left|\frac{1}{F^{{\prime}^2}}\right|$ are harmonic functions. So if $u=\log\left|\frac{1}{F^{{\prime}^2}}\right|$, then both $u, e^u$ are harmonic functions. Therefore $\frac{\partial^{2} e^u}{\partial x^2}+\frac{\partial^{2} e^u}{\partial y^2}=0$.  This implies $\left(\frac{\partial u}{\partial x}\right)^2+\left(\frac{\partial u}{\partial y}\right)^2=0$ and hence $F^{\prime}$ is identically constant (say $k$). Therefore $F(z)=kz+l$ and $\operatorname{Re} \left(\frac{FF^{\prime\prime}}{{F^{\prime}}^2}\right)=0$. From equations \eqref{zeroeqn} and \eqref{zeroeqntimelike}, we get $|k|=1,\;|k|<1$ and $|k|=1,|k|>1$ respectively. Therefore, we have arrived at contradictions in both these cases. Hence, the proof. 
\end{proof}
\begin{remark}\label{mainthmremark}
    A similar idea can be used to show that there does not exist any non-planar minimal modular surface in $\mathbb{E}^3$. 
    
\end{remark}


\section{Further Generalizations and Applications}\label{application}

    Define $(M^m,g)$ to be an oriented semi-riemannian manifold. Define $h:M\to\mathbb{R}^+$ to be a log-harmonic function (i.e., $\Delta_g\mathrm{log}\hspace{0.05cm}h=0$ where $\Delta_g$ is the laplace-beltrami operator corresponding to $g$). Choosing local coordinates $\{x^1,\dots,x^m\}$ on $M$, we have
    \begin{equation}\label{eq6.1}
        \begin{aligned}
            0 &=\Delta_g\mathrm{log}\hspace{0.05cm}h=\frac{1}{\sqrt{|g|}}\partial_i(\sqrt{|g|}g^{ij}\partial_j(\mathrm{log}\hspace{0.05cm}h))=\frac{1}{\sqrt{|g|}}\partial_i(\sqrt{|g|}g^{ij}\frac{1}{h}\partial_jh)\\
            &=\frac{1}{h}\frac{1}{\sqrt{|g|}}\partial_i(\sqrt{|g|}g^{ij}\partial_jh)-\frac{1}{h^2}g^{ij}\partial_ih\partial_jh=\frac{1}{h}\Delta_gh-\frac{1}{h^2}|\nabla_gh|^2.
        \end{aligned}
    \end{equation}
    From equation \eqref{eq6.1}, $\Delta_gh=\frac{1}{h}|\nabla_gh|^2$ is invariant of conformal change of the metric $g$. As a result, by choosing a conformal class $[g]$ and $(M,[g])$ appropriately as a Riemann surface, we land with an alternate proof of the facts (Theorem \ref{mainthm}, Remark \ref{mainthmremark}) that there aren't any non-planar maximal (resp. minimal) modular surfaces in $\mathbb{E}_1^3$ (resp. in $\mathbb{E}^3$).
    
\subsection{CMC Modular Surfaces}
    
    Given a log-harmonic function $h:U\overset{\mathrm{open}}{\subset}\mathbb{R}^2\to\mathbb{R}^+$, suppose we assume that the local immersion $\phi(x',y'):=(x',y',h(x',y'))$ in $\mathbb{E}^3$ (or the spacelike immersion in $\mathbb{E}_1^3$) is isometric and the local coordinates $(x',y')$ are chosen to be orthonormal (respectively spacelike orthonormal) with respect to the pullback metric $\phi^*ds^2$ (where $ds^2$ is the euclidean metric on $\mathbb{E}^3$ and respectively the metric on $\mathbb{E}_1^3$). Lastly, assume that the immersion is of constant mean curvature (CMC) vector with the scalar $H$ being its third coordinate (as a vector in $\mathbb{E}^3$). Then the following two identities (due to equation 2.43 \cite{chen} and log harmonicity of $h$) characterize such an immersion:
    \begin{equation}\label{eq6.2}
        \Delta h:=\partial_{x'x'}h+\partial_{y'y'}h=2H,
    \end{equation}
    \begin{equation}\label{eq6.3}
        |\nabla h|^2:=(\partial_{x'}h)^2+(\partial_{y'}h)^2=2Hh.
    \end{equation}
    \begin{theorem}\label{thm5.1}
        The only class of non-zero CMC modular surfaces in $\mathbb{E}^3$ (or non-zero CMC spacelike modular surfaces in $\mathbb{E}_1^3$) are given locally by the height function $h(x',y')=\frac{H}{2}((x')^2+(y')^2)+ax'+by'+\frac{a^2+b^2}{2H}$, where $(x',y')$ are local orthonormal coordinates on the surface with respect to the pulled back metric $\phi^*ds^2$.
    \end{theorem}
    \begin{proof}
        In a succinct manner, we shall show that (both the) general class of solutions of equation \eqref{eq6.2} which satisfy equation \eqref{eq6.3} and vice versa have precisely the aforementioned height function. The general solution of equation \eqref{eq6.2} is given by $h_{\mathrm{gen}}=h_p+h_{\mathrm{hom}}$ (where $h_p$ is a particular solution to equation \eqref{eq6.2} and $h_{\mathrm{hom}}$ is the general solution to the homogenous equation $\Delta h=0$). We can choose $h_p(x',y')=\frac{H}{2}((x')^2+(y')^2)$ and let $h_{\mathrm{hom}}=f$ be any harmonic function. Plugging $h_{\mathrm{gen}}$ into equation \eqref{eq6.3}, we get a nonlinear first order constraint on $f$: 
        \begin{equation}\label{eq6.4}
            (\partial_{x'}f)^2+(\partial_{y'}f)^2+2H(x'\partial_{x'}f+y'\partial_{y'}f-f)=0.
        \end{equation}
        Solving equation \eqref{eq6.4} using method of characteristics, we get its general solution to be
        \begin{equation}
            f(x',y')=ax'+by'+\frac{a^2+b^2}{2H},
        \end{equation}
        for real constants $a,b$. Thus the general class of solutions of equation \eqref{eq6.2} which solve equation \eqref{eq6.3} are of the form
        \begin{equation}\label{eq6.5}
            \tilde{h}(x',y')=\frac{H}{2}((x')^2+(y')^2)+ax'+by'+\frac{a^2+b^2}{2H}.
        \end{equation}
        Next we find the general class of solutions of equation \eqref{eq6.3} (using method of characteristics) to be
        \begin{equation}\label{eq6.6}
            \overline{h}(x',y')=\frac{H}{2}[(x'-x_0)^2+(y'-y_0)^2]+h_0,
        \end{equation}
        where $x_0,y_0,h_0$ are real constants. Moreover it can be seen that the function given in equation \eqref{eq6.6} also satisfies equation \eqref{eq6.2}.\\
        Lastly checking that both equations \eqref{eq6.5} and \eqref{eq6.6} depict the same type of height functions and naturally choosing the more constrained expression to be the most general solution of the two equations, we are through. 
    \end{proof}
    On a closing note, we make the following clarification.
    \begin{remark}
        Although the height function $h(x',y')$ given in Theorem \ref{thm5.1} appears to be that of a paraboloid of revolution, it doesn't necessarily conform to that class. That is because the local coordinates $(x',y')$ in which the height function is expressed happens to be very special (orthonormal in the pulled back metric $\phi^*ds^2$) and isn't necessarily the same as the axis coordinates $(x,y)$ used in all of the other parts of the paper. 
    \end{remark}

\subsection{Asymptotic Behaviour of Convex Modular Surfaces}

    We define a convex modular surface locally given as a graph $(x,y,h(x,y))$ such that $\mathrm{det}(\mathrm{Hess})=h_{xx}h_{yy}-h_{xy}^2>0$ (where $\mathrm{Hess}$ is the hessian of $h$). Let $S\subset\mathbb{E}^3$ be a convex modular surface defined on a simply connected region $U\overset{\mathrm{open}}{\subset}\mathbb{R}^2$. As usual, we assume the log-harmonic condition for $h:U\to\mathbb{R}^+$ (i.e., $h\Delta h=|\nabla h|^2$, as in Fact \ref{fact2}). Then the Gaussian curvature $K$ of $S$ is given by the following Monge-Ampere equation (where we have substituted $|\nabla h|^2=h\Delta h$);
    \begin{equation}\label{eq6.12}
        K=\frac{\mathrm{det}(\mathrm{Hess})}{(1+h\Delta h)^2}.
    \end{equation}
    The general solution for $K\equiv 0$ (i.e., solutions of the homogenous Monge-Ampere equation) has already been considered in Section \ref{CMS}. Since in our setup, $\mathrm{det}(\mathrm{Hess})>0$ (by convexity), we have by equation \eqref{eq6.12} $K>0$. Expanding \eqref{eq6.12}: \begin{equation}\label{eq6.13}
        h_{xx}h_{yy}\geq h_{xx}h_{yy}-h_{xy}^2=K(1+h\Delta h)^2>Kh^2(\Delta h)^2\geq 2Kh^2h_{xx}h_{yy}> 0.
    \end{equation}
    In inequation \eqref{eq6.13}, we have used the well-known inequality $(\Delta h)^2=(h_{xx}+h_{yy})^2\geq 2h_{xx}h_{yy}$. From \eqref{eq6.13}, we have
    \begin{equation}\label{eq6.14}
        K<\frac{1}{2h^2}.
    \end{equation}
    Now in the case where the region of definition of the convex modular graph was defined to be the whole of $\mathbb{R}^2$ and using the Liouville's theorem for subharmonic functions as well as maximum principle for subharmonic functions at infinity, we have the following result:
    \begin{prop}\label{prop5.1}
        For a complete\footnote{The exponential map defined on the tangent space of any point of the surface is a diffeomorphism onto $\mathbb{R}^2$.} convex modular graph in $\mathbb{E}^3$, its Gaussian curvature $K\to 0$ as $|(x,y)|_{\mathbb{R}^2}\to\infty$, where $|(\cdot,\cdot)|_{\mathbb{R}^2}$ is the euclidean metric norm.
    \end{prop}
    Also as an interesting consequence of inequation \eqref{eq6.14}, we have
    \begin{prop}\label{prop5.2}
        There exists no constant Gaussian curvature complete convex modular graph in $\mathbb{E}^3$ or equivalently, there isn't any constant positive Gaussian curvature modular graph in $\mathbb{E}^3$.
    \end{prop}
    Next, we consider the previous exact setup when $S\subset\mathbb{E}_1^3$ is a spacelike convex modular surface (i.e., on the domain of definition, $1>|\nabla h|^2\geq 0$). Its Gaussian curvature is denoted by equation \eqref{eq3}:
    \begin{equation}
        K=-\frac{h_{xx}h_{yy}-h_{xy}^2}{(1-|\nabla h|^2)^2}.
    \end{equation}
    Define $0< K':=-K$. Then,
    \begin{equation}
        K'\geq\frac{h_{xx}h_{yy}-h_{xy}^2}{(1-|\nabla h|^2)(1+|\nabla h|^2)}=\frac{h_{xx}h_{yy}-h_{xy}^2}{1-|\nabla h|^4}.
    \end{equation}
    Next, we multiply both sides by $1-|\nabla h|^4$, substitute the log harmonic condition $h\Delta h=|\nabla h|^2$ and use the inequality $(\Delta h)^2=(h_{xx}+h_{yy})^2\geq 2h_{xx}h_{yy}$ to get
    \begin{equation}
        K'-2K'h^2h_{xx}h_{yy}\geq K'-K'h^2(\Delta h)^2\geq h_{xx}h_{yy}-h_{xy}^2> 0.
    \end{equation}
    As a result, we have
    \begin{equation}
        h_{xx}h_{yy}<\frac{1}{2h^2}.
    \end{equation}
    Thus using Liouville theorem and Maximum Principle at infinity for subharmonic functions, we have $h_{xx}h_{yy}\to 0$ as $|(x,y)|_{\mathbb{R}^2}\to\infty$. By convexity condition, we have $0\leq h_{xy}^2<h_{xx}h_{yy}\implies h_{xy}\to 0$ as $|(x,y)|_{\mathbb{R}^2}\to\infty$ and by spacelike condition; $0<1-|\nabla h|^2=1-h(h_{xx}+h_{yy})\implies 0<h_{xx}+h_{yy}<\frac{1}{h}$. Thus (by the above stated properties of subharmonic functions) $h_{xx},h_{yy}\to 0$ as $|(x,y)|_{\mathbb{R}^2}\to\infty$; giving us the following result:
    \begin{prop}\label{prop5.3}
        The height function of any complete convex spacelike modular graph in $\mathbb{E}_1^3$ has asymptotically vanishing hessian. 
    \end{prop}
    
\subsection{Modular Surfaces as local solutions to Euler-Lagrange equations of Liouville Conformal Field Theory}\label{cft}

    Suppose we define a log-harmonic function $h:M\to\mathbb{R}^+$. The function $h$ can also be seen as a scalar field on $M$. Now define $(M,[g])$ to be a Riemann surface (we use here the definition of a Riemann (Lorentz) surface as a conformal riemannian (lorentzian) oriented 2-manifold). The action of the Liouville CFT (equation 1.1, \cite{liouville}) and the corresponding Euler-Lagrange equation are respectively the following:
    \begin{equation}\label{eq6.8}
        S_L[\phi, g] = \frac{1}{4\pi} \int d^2x \, \sqrt{g} \left( g^{ab} \partial_a \phi \, \partial_b \phi + Q R \phi + 4\pi \mu e^{2b\phi} \right),
    \end{equation}
    \begin{equation}\label{eq6.9}
        \Delta_g\phi-QS_g\phi+2\mu e^{2b\phi}=0,
    \end{equation}
    where $\Delta_g$ is the laplace-beltrami operator corresponding to $g$, $S_g$ is the scalar curvature of $(M,g)$, $Q = b + \frac{1}{b}$ is the background charge, $\mu$ is the cosmological constant and $b$ is the coupling constant.\\
    It can be seen that this equation is conformally invariant in the sense that its solution space doesn't depend on the choice of $g\in [g]$. Substituting $h$ in place of $\phi$ in equation \eqref{eq6.9} and using $\Delta_gh=\frac{1}{h}|\nabla_gh|^2$, we get a first order nonlinear partial differential equation, the solutions (local integral surfaces) of which will be modular surfaces. As a last note, when $[g]$ is a conformal class of riemannian metrics it can be observed that the non-triviality of the local solutions (i.e. when the modular surfaces are NOT of the form $(x,y,h(x,y))$ where $h\equiv\mathrm{constant}$) can be guaranteed in the case of Liouville CFTs if and only if $QS_gh-2\mu e^{2bh}>0$.

\subsection{Timelike Modular Surfaces as local solutions to Euler-Lagrange equations of Massless Scalar CFT}\label{massless}

    We define a timelike modular surface as given locally by the graph $\phi(x,y):=(x,y,h(x,y))$ ($\phi$ is a timelike immersion into $\mathbb{E}_1^3$ from a Lorentz surface). Now the action and the corresponding Euler-Lagrange equation of a massless scalar CFT (see Section 3.2 \cite{quantum_book}) are respectively given by
    \begin{equation}\label{eq6.x}
        S[\phi] = -\frac{1}{2} \int d^2x \; \eta^{\mu\nu} \, \partial_\mu \phi \, \partial_\nu \phi,
    \end{equation}
    \begin{equation}\label{eq6.y}
        \Box \phi = \phi_{xx}-\phi_{yy} = 0,
    \end{equation}
    where $\Box$ is the D'Alembert's operator. Now substituting $h$ in place of $\phi$ in equation \eqref{eq6.y} and using $\Box h=\frac{1}{h}|\nabla h|^2$, we get a first order linear partial differential equation, certain solutions for which can be log-harmonic functions whose gradients are locally lightlike. It can be checked that for any $C^2$ function $f:\mathbb{R}\to\mathbb{R}^+$, $h(x,y):=f(x+y)$ or $h(x,y):=f(x-y)$ satisfies these criteria and constitute the solution set of \eqref{eq6.y}.
    
\subsection{Modular Height Functions and Sigma Models}\label{sigma}

    Suppose $U\overset{\mathrm{open}}{\subset}(M^2,g)$ be an open subset of a semi-riemannian 2-manifold and let $h:=(h^1,\dots,h^n):(U,g|_U)\to ((\mathbb{R}^n)^+,\tilde{g})$ (where $\tilde{g}$ is a riemannian metric over $(\mathbb{R}^n)^+:=\underbrace{\mathbb{R}^+\times\dots\times\mathbb{R}^+}_{\mathrm{n\hspace{0.1cm}times}}$). For $h$ to be a sigma model, its components $h^i$ must satisfy the Euler-Lagrange equations (see Section 5.1 \cite{quantum_book})
    \begin{equation}\label{eq6.10}
        \Delta_gh^i+\Gamma^i_{jk}(h)\tilde{g}^{\mu\nu}\partial_{\mu}h^j\partial_{\nu}h^k=0.
    \end{equation}
    Assume further that each $h^i$ is log-harmonic and $g$ is semi-euclidean. Then we have the following simplified (first order pde) form of equation \eqref{eq6.10}:
    \begin{equation}\label{eq6.11}
        F^i(h,p,q):=\pm(p^i)^2\pm(q^i)^2+\Gamma^i_{jk}(h)(p^jp^k+q^jq^k)h^i=0,
    \end{equation}
    where $\Gamma^i_{jk}(h)$ are Christoffel symbols of $\tilde{g}$ along image of $h$ and $p^i:=\partial_xh^i,q^i:=\partial_yh^i$ and the signs $\pm$ are determined by the signature of the semi-euclidean metric $g$.\\
    Given any $n$-tuple $(h^1,\dots,h^n)$ of modular height functions on $U$, we get a corresponding map $h$ from $(U,g|_U)$ to $(\mathbb{R}^n)^+$. Substituting $A^i_{jk}:=(p^jp^k+q^jq^k)h^i$ and $B^i:=\pm(p^i)^2\pm(q^i)^2$ in equation \eqref{eq6.11}, we have a set of $n$ linear equations in $n^3$ unknowns (i.e., all functions $\Gamma^i_{jk}(h)$, $\forall\hspace{0.05cm}1\leq i,j,k\leq n$). This system of equations being underdetermined has plenty of solution sets $\{\Gamma^i_{jk}|\hspace{0.05cm}1\leq i,j,k\leq n\}$, each of which (when smoothly extended to $(\mathbb{R}^n)^+$ from $\mathrm{Image}(h)$) correspond to the set of admissible geometries $\tilde{g}$ on $(\mathbb{R}^n)^+$ which make the map $h$ a sigma model.
    
\section{Acknowledgement}
    The first author is funded by CSIR-JRF and SRF grants (File no. 09/1023(12774)/2021-EMR-I). The second and fourth authors of this paper acknowledge the funding received from UGC, India.  The third author is partially supported by the MATRICS grant (File No. MTR/2023/000990), which has been sanctioned by the SERB. 
    
\section{Declarations}
    \begin{itemize}
          \item \textbf{Funding:} Not applicable.
          \item \textbf{Author Contributions: } All the authors equally contributed to the content and to the writing of the paper.
          \item \textbf{Data Availability:} Not applicable
          \item \textbf{Conflict of interest:} There are no conflicts of interest to declare.
    \end{itemize}

\end{document}